\documentclass{amsproc}

\usepackage{amssymb,graphicx,rotating,multirow,multicol}
\usepackage{amsmath,amsthm,amsfonts,amscd}
\usepackage[toc,page]{appendix}
\usepackage{caption,comment}
\usepackage[mathscr]{euscript}
\usepackage{epstopdf}
\usepackage{float}

\usepackage{graphicx}

\newtheorem{theorem}{Theorem}[section]
\newtheorem{lemma}[theorem]{Lemma}
\newtheorem{cor}[theorem]{Corollary}

\theoremstyle{definition}

\theoremstyle{remark}
\newtheorem{remark}[theorem]{Remark}

\def\ie{\emph{i.e.}}

{\catcode`\@=11 \gdef\mnote#1{\marginpar{\footnotesize
 \tolerance\@M\spaceskip2.6\p@ plus10\p@ minus.9\p@\rm#1}}}
\def\Dg:{\endgraf{\bf Dg:\enspace}\ignorespaces}

\let\Bbb\mathbb
\let\Cal\mathcal

\newcommand{\be}{\begin{equation}}
\newcommand{\ee}{\end{equation}}
\let\ge\geqslant 
\let\le\leqslant 

\def\Z{\Bbb Z}
\def\R{\Bbb R}
\def\C{\Bbb C}
\def\Q{\Bbb Q}

\let\+\sqcup

\numberwithin{equation}{section}

\begin{document}

\title[Anti-symplectic involutions on rational symplectic 4-manifolds]
{Anti-symplectic involutions on rational symplectic 4-manifolds}


\author{V.~Kharlamov}
\address{Universit\'{e} de Strasbourg et IRMA (CNRS),
7 rue
Ren\'{e}-Descartes, 67084 Strasbourg Cedex, France}
\email{kharlam@math.unistra.fr}

\author{V.~Shevchishin}
\address{University of Warmia and Mazury,
ul.~Słoneczna 54, 10-710 Olsztyn, Poland}
\email{vsevolod@matman.uwm.edu.pl}


\date{}

\begin{abstract}
This is an expanded version of the talk given by the first author at the conference "Topology, Geometry, and Dynamics: Rokhlin – 100". 
The purpose of this talk was to explain our current results on classification of rational symplectic 4-manifolds equipped with an anti-symplectic involution. 
Detailed exposition will appear elsewhere.
\end{abstract}

 \dedicatory{Dedicated to the memory of V.A.~Rokhlin}
\maketitle


\section{Introduction} 
This paper deals with such questions as classifying rational symplectic 4-mani\-folds equipped with an anti-symplectic involution and existence of an equivariant K\"ahler structure on these manifolds.

Our motivation was twofold. Firstly, it is a natural, necessary, step in generalizing real algebraic geometry 
achievements on the range of problems in the spirit of Hilbert's 16th problem into the realm of symplectic geometry; 
especially, in what concerns the study of interactions between topological and deformation equivalence invariants. Secondly, 
it is very closely connected with the contemporary study of real analogs of Gromov-Witten invariants like those discovered by
J.-Y.~Welschinger \cite{W} in early 2000s.

By a {\it rational symplectic manifold} we mean a symplectic manifold which can be obtained by a sequence of symplectic blow-ups and blow-downs from $\C P^2$ and $\C P^1\times \C P^1$ equipped with their standard symplectic structures $\lambda \omega_{FS}$ and $\lambda \omega_{FS}\times \mu\omega_{FS} $, 
where $\omega_{FS}$ states for symplectic structures given by Fubini-Study metrics (see \cite{McD-P} for the definition and basic properties of symplectic blow-ups and blow-downs). 
As is known \cite{McD}, the rational symplectic 4-manifolds have no other blow-down minimal models than
$(\C P^2, \lambda \omega_{FS})$ and $(\C P^1\times \C P^1, \lambda \omega_{FS}\times \mu\omega_{FS}) $, and, furthermore, each rational symplectic 4-manifold can be obtained from a minimal model by a simultaneous blow-up of a finite collection of disjoint embedded balls.

By an {\it anti-symplectic involution} on a symplectic manifold $(X,\omega)$ we call a
diffeomorphism $c :X\to X$ such that $c^2=\operatorname {id}_X$ and $c^*\omega=-\omega$. We name such a triple $(X,\omega, c)$
a {\it real symplectic manifold} and say
that it is {\it K\"ahler} if there exists an
integrable complex structure $J$ 
with respect to which $c$ is an anti-holomorphic involution ({\it a real structure}) and
$\omega$ a K\"ahler form on $(X,J)$. If $(X,\omega, c)$ does not admit any $c$-equivariant blow-down, we call it {\it $c$-minimal.}

We say that $(X, \omega_X, c_X)$ is {\it $c$-equivariant deformation equivalent} to $(Y, \omega_Y, c_Y)$, if
there exists a chain of triples $(X_i,\omega_i, c_i)$, $i=0,\dots, k$, such that $(X, \omega_X, c_X)$ is isomorphic to 
$(X_0,\omega_0, c_0)$, $(Y, \omega_Y, c_Y)$ to $(X_k,\omega_k, c_k)$, and $(X_i,\omega_i, c_i)$, for each $1\le i \le k$, is isomorphic
to a triple obtained by a smooth $c_{i-1}$-equivariant variation $\omega_{i-1}(t)$ of $\omega_{i-1}$ on $X_{i-1}$.

Our first result is concerned with existence of K\"ahler structure.

\begin{theorem}\label{main}  Every real rational symplectic 4-manifold $(X,\omega, c)$ is $c$-equivari\-ant deformation equivalent to a K\"ahler one.
If, in addition, $(X,\omega, c)$ is $c$-minimal, then $(X,\omega, c)$ itself is K\"ahler.
\end{theorem}

According to Kodaira embedding theorem, each compact K\"ahler surface is deformation equivalent to a projective one. The equivariant version of this theorem 
and $c$-equivariant deformation unicity of K\"ahler structures lead immediately to our next result.

\begin{theorem} Deformation classification of real rational symplectic 4-manifolds coincides with deformation classification of real rational algebraic surfaces. \qed
\end{theorem}

As is proved in \cite{DK}, two rational real algebraic surfaces are deformation equivalent if, and only if, their complex conjugation involutions are diffeomorphic.
Thus, we get as an easy consequence the following result.

\begin{theorem} Real rational symplectic 4-manifolds $(X, \omega_X, c_X)$ and $(Y, \omega_Y, c_Y)$ are $c$-equivariant deformation equivalent if, and only if,
the involutions $c_X$ and $c_Y$ are diffeomorphic. \qed
\end{theorem}

\section{Tools} The tools that play the key role in the proofs of our main results are the equivariant versions of: (1) the technique of $J$-holomorphic curves invented by M.~Gromov, (2) the relations between Gromov-Witten and Seiberg-Witten invariants established by C.~Taubes and enhanced by T.~J.~Li and A.Liu, and (3) the inflation technique invented by F.~Lalonde and D.~McDuff. Transition to a version equivariant with respect to an anti-symplectic involution is often rather easy and straightforward, but in certain cases (and this concerns especially applications of the inflation technique) it rises additional difficulties caused by the need to understand codimension-one events.

To formulate principal technical statements, we need to fix a few definitions and notations.

We denote by $\Omega (X)$ the space of symplectic forms on $X$, by
$\Cal C=\Cal C(X)\subset H^2(X,\R)$ the set of cohomology classes representable by
symplectic form, and for an element $A$ or a subset $U\subset\Cal C(X)$ we
denote by $\Omega (X,A)$ and resp.\ by $\Omega(X,U)$ the subspace of symplectic
forms whose cohomology class equals $A$ or resp.\ lies in $U$.  In the case
when $c:X\to X$ is an~anti-involution we denote by $\Omega (X,c)$,
$\Omega (X,A,c)$, and resp.\ $\Cal C(X,c)$, the spaces of
$c$-anti-invariant forms (in the given cohomology class $A$), and the
space of their cohomology classes.

For a symplectic form $\omega$ we write $\Omega(X,\omega)$ instead of $\Omega(X,[\omega])$, and
$\Omega(X, \omega, c)$ instead of $\Omega(X,[\omega], c)$. Denote by
$\Omega_0(X,\omega, c)$ (resp.\ by $\Omega_0(X,\omega)$) the connected component of the
space $\Omega(X,\omega, c)$ (resp.\ of $\Omega(X,\omega)$) containing the form $\omega$, and
by $\pi_0\Omega(X, \omega, c)$ (resp.\ by $\pi_0\Omega (X,\omega)$) the whole set of connected
components of $\Omega(X, \omega, c)$ (resp.\ by $\pi_0\Omega (X,\omega)$).

By $\Cal E= \Cal E(X)\subset H_2(X,\Z)$ we denote the set of homology classes represented by
smoothly embedded 2-spheres with self-intersection $-1$ (that are called {\it exceptional spheres}).

\begin{theorem}[Symplectic cone of rational symplectic $4$-manifolds \cite{Li-Liu}] \label{SyCone} 
 Let $X$ be a rational 
 symplectic $4$-manifold and
$\Omega\in H^2(X,\R)$ a cohomology class. Then $\Omega$ is represented by a symplectic
form 
if and only if it satisfies the following conditions:
 \begin{itemize}
\item[\rm($\Omega0$)] 
$\int_X\Omega^2>0$; 
\item[\rm($\Omega1$)] 
$\int_E\Omega\neq0$ for
 every exceptional sphere $E\subset X$.
\end{itemize}

Furthermore, two symplectic forms $\omega_0,\omega_1$ have equal Chern classes,
$c_1(X,\omega_0)=c_1(X,\omega_1)$, if and only if there exists a family
$\Omega_t\in H^2(X,\R)$ of cohomology classes which satisfy the conditions
$(\Omega0,\Omega1)$ and connect the classes of the forms, \ie, $\Omega_0=[\omega_0]$ and
$\Omega_1=[\omega_1]$. 
\end{theorem}

Given a symplectic manifold $(X,\omega)$  with an $\omega$-tamed almost complex structure $J$, we call the {\it curve cone} the cone in $H^2(X; \R)$ generated
by the classes Poincar\'e dual to $J$-holomorphic curves in $X$, and denote this cone by $C_{X,J}$. When, in addition, $(X,\omega)$ is equipped with an anti-involution $c$ and $J$ is $c$-anti-invariant, we define the {\it real curve cone} $\R C_{X,J}$ as the cone generated by Poincar\'e duals of $c$-invariant $J$-holomorphic curves (including combinations $D+cD$ !). We put  
$$\R C_{X,J}^{KD\ge 0} =\{ D\in \R C_{X,J} : KD\ge 0\}.$$

We say that $(X, J, c)$ (respectively, $(X,\omega, J, c)$) is {\it real rational ruled}, if it is equipped with a {\it $c$-invariant $J$-holomorphic ruling}: that is a $c$-equivariant smooth map $X\to S^2$ with $J$-holomorphic  fibers and $J$-holomoprhic sphere as a generic fiber. 

\begin{theorem}[Real symplectic version of Mori theorem, {\it cf.} \cite{Z}]\label{Mori}
Let $(X,\omega)$ be a closed symplectic manifold equipped with an anti-symplectic involution $c$ and an $\omega$-tamed $c$-anti-invariant almost complex structure $J$.
Then, there exist countably many smooth irreducible curves $L_i$ with $L_i^2=-1, 0$, or $1$ such that $L_i + cL_i$ generate extremal rays in the cone $\R C_{X,J}$ of real curves on $(X,J,c)$
and
$$
\R C_{X,J} = \R C_{X,J}^{KD\ge 0} + \sum\R^+ (L_i + cL_i).
$$
If besides these $L_i$ there is one with $L_i^2=1$ (respectively, $L_i^2=0$), then $X$ is $\C P^2$ and $c$ is  $c_{st}$ (respectively, $(X,J, \omega, c)$ is a $c$-minimal real 
symplectic manifold with a $c$-invariant $J$-holomorphic ruling having  $L_i$ as a fiber).
\end{theorem}

\begin{lemma}[Finding $c$-equivariant embedded curves,  {\it cf.} \cite{McD3} ]  Let $(X,\omega, c)$ be a rational $4$-manifold
 with anti-involution, $J$ a generic $c$-anti-invariant $\omega$-tamed almost complex structure, and $A\in H_2(X,\Q)$ a rational homology class
 satisfying $c_*A=-A$, $A^2>0$ and $A\cdot[E]>0$ for every 
 exceptional sphere $E$. Then some multiple $2mA$ is represented by a
 $\omega$-symplectic $c$-invariant $J$-holomorphic curve $\Sigma$.  
 \end{lemma}

\begin{lemma}[Real version of Inflation lemma, {\it cf.} \cite{McD2}] Let $J$ be a $c$-anti-invariant $\omega$-tamed almost complex structure on a real symplectic manfiold
$(X, \omega, c)$ that admits a real non-singular $J$-holomoprhic curve $Z$ with $Z\cdot Z\ge 0$. Then, there exists a family $\omega_t, t\ge 0$, of $c$-anti-invariant symplectic forms that all tame $J$ and have cohomology class $[\omega_t]=[\omega]+ t \,{\operatorname {PD}}(Z)$.
\end{lemma}

These two lemmas are crucial for the proof of the next theorem. (For purposes of this paper, it would be sufficient to reduce the statement of this theorem to the case $Y=[0,1]$ and $Y'=\partial Y$.)

\begin{theorem}[Homotopic symplectic Nakai-Moishezon theorem]\label{NaMo-Omega} 
Let $Y$ be a compact manifold possibly with boundary,
$\{(X, \psi_y, J_y)\}_{y\in Y}$ a family of rational symplectic manifolds, $\{\Omega_y\}_{y\in Y}$ a continuous family of $\R$-valued $2$-cohomology  classes on $X$,
and $\{\omega'_y\}_{y\in Y'}$, for a certain closed subset $Y'\subset Y$, a continuous family of symplectic forms on $X$.
 Assume that
the following conditions are satisfied: 
 \begin{itemize}
 \item[\rm(NM0)] 
There exists a 
continuous map $Y\times I\to {\Cal C}(X)\subset H^2(X,\R), (y,t)\mapsto A_{y,t},$  deforming
the classes $[\psi_y]=A_{y,0}$ into the classes $\Omega_y=A_{y,1}$.
\item[\rm(NM1)] 
 For each $y\in Y$ and 
each 
 irreducible $J_y$-holomorphic curve $C$ with\, $[C]^2<0$, one
 has $\int_C\Omega_y>0$;
\item[\rm(NM2)] 
For each $y\in Y'$, the form
 $\omega'_y$ tames $J_y$ and has the cohomology class $\Omega_y$.
\end{itemize} 

Then there exists a family $\{\omega_y\}_{y\in Y}$ of symplectic forms on $X$ 
such that $\omega_y=\omega'_y$ for $y\in Y'$ and $[\omega_y]=\Omega_y$ for $y\in Y$.
Moreover, there exist families $\{\tilde J_y\}_{y\in Y}$ , 
tamed by
$\{\omega_y\}_{y\in Y}$ which are arbitrarily close to the family
$\{J_y\}_{y\in Y}$.

Furthermore, if $c:X\to X$ is an anti-involution such that the families
$\{J_y\}_{y\in Y}$, $\{\Omega_y\}_{y\in Y}$, and $\{\omega'_y\}_{y\in Y'}$ are 
$c$-anti-invariant, then the family $\{\omega_y\}_{y\in Y}$
also can be chosen $c$-anti-invariant.
\end{theorem}

\section{Proof Outline}
We start from classifying $c$-minimal models. Applying Theorem \ref{Mori}, and using so-called Gromov's automatic regularity (see \cite{HLS})
for constructing rulings,
we get the following result.

\begin{theorem}
Let 
$(X,\omega)$ be a non-minimal rational symplectic $4$-manifold
 with an~anti-involution $c$, and let $J$ be an~$\omega$-tame $c$-anti-invariant almost
 complex structure.
 For each exceptional $J$-holomorphic curve $E$,
 set $E'=c(E)$ and orient $E'$ as a $J$-holomorphic curve.  
Then:
 \begin{enumerate}
 \item The minimal value of 
 $[E]\cdot[E']$ 
 lies between $-1$ and
 $3$.
\item If $[E]\cdot[E']=-1$,  the exceptional curve $E$ is
$c$-invariant 
and there exists a~symplectic
$4$-manifold with an~anti-involution $(X_1,\omega_1,c_1)$ 
such that $(X,\omega, c)$ is the result of
the symplectic blow-up of $(X_1,\omega_1, c_1)$ performed in a $c_1$-invariant symplectic ball $B$.
\item  If $[E]\cdot[E']=0$, the  curves $E$ and $E'$ are
disjoint, and there exists a~symplectic $4$-manifold with
an~anti-involution $(X_1,\omega_1,c_1)$ and disjoint symplectic balls $B,B'\subset X_1$
with $B'=c_1(B)$, such that $(X,\omega, c)$ is the result of the symplectic
blow-up of $(X_1,\omega_1, c_1)$ performed in 
$B,B'$.
\item  If $[E]\cdot[E']=1$, there exists a unique $J$-holomorphic
ruling $pr:X\to Y=\C P^1$
such that $E\cup E'$ is a fiber of this
ruling. Moreover, 
$pr\circ c= c_{st}\circ pr$ and $y^*= pr(E\cup E')$ is a fixed point of
$c_{st}$, where $c_{st} : \C P^1\to \C P^1$ is the standard complex conjugation involution. 
\item  
If there are no any exceptional curve $C$ with $[C]\cdot [C'] <2$, then
either $b_2(X)= 7$, $K_X^2=2$, and there exists an exceptional curve $E$ such that
$E+E' = -K_X$, or $b_2(X)= 8$, $K_X^2=1$, and there exists an exceptional curve $E$ such that
$E+E' = -2K_X$. In the both cases the curve cone $\R C_{X,J}$ is generated by $K_X$.
\qed 
\end{enumerate}
\end{theorem}

\begin{cor}\label{cases}
For $c$-minimal real rational symplectic 4-manifolds, the list of possible underlying manifolds $X$ and $X^c=\operatorname{Fix} c$  is as follows:
\begin{enumerate}
\item $X=\C P^2$ and $ X^c=\R P^2$.
 \item $X= \C P^1\times \C P^1$ and    $X^c$ is either $S^1\times S^1$, $S^2$, or $\emptyset$.
\item $X$ 
admits a $c$-equivariant 
rational ruling
with $2m$-singular fibers, $m>1$, and $X^c= m S^2$. 
\item $X$
has the same homology as
a del Pezzo surface with $K^2=2$, the anti-involution $c$ acts in $H_2(X)$ as a reflection in $K$, and $X^c= 4S^2$.
\item $X$ 
has the same homology as
a del Pezzo surface with $K^2=1$,  the anti-involution $c$ acts in $H_2(X)$ as a reflection in $K$, and $X^c= 4S^2\sqcup\R P^2$.
\qed
\end{enumerate}
\end{cor}

Thus, except four special types of small topological complexity (items (1), (2), (4), (5)), each of other $c$-minimal real rational symplectic 4-manifolds is rational ruled.
Our proof strategy for Theorem \ref{main} is to treat first the case of rational ruled 4-manifolds, and then to deduce from it the remaining cases.

For $(X,\omega, c)$ with a $c$-equivariant ruling $pr: X\to S^2$ 
as in item (3) of Corollary \ref{cases}
we observe, first, that:
\begin{itemize}
\item Up to a $c$-equivariant diffeomorphism, the mapping $pr: X\to S^2$ is uniquely defined by $m$.
\item The $c$-anti-invariant part of $H_2(X)$ is generated by 2 elements, $F$ and $H$, where $F$ is the class of the fiber, $F^2=0$, $H^2=0$,
and $F\cdot H$ is equal to $2$ for $m$ even and $4$ for $m$ odd.
\item The cone of 
 $c$-anti-invariant symplectic classes $[\omega]$ is given by $aF+bH$ with $a,b>0$.
\item When $(X, c, pr)$ is a generic algebraic real conic bundle, this cone coincides with the 
cone of $c$-anti-invariant K\"ahler 
classes of $X$.
\end{itemize}

After that we are left to prove the following theorem.

\begin{theorem}\label{main-conic}
 Let $(X,c,pr)$ be a 
 $c$-equivariant ruling whose fibers are pseudo-holomorphic with respect to two $c$-anti-invariant almost complex structures, $J$ and $J'$,
tamed by two $c$-anti-invariant symplectic structures, $\omega_0$ and $\omega_1$, respectively. If $\omega_0$ and $\omega_1$ represent the same cohomology class, then they
are $c$-equivariant isomorphic.
\end{theorem}

\begin{proof}
We start from making the almost complex structures identical and integrable at a neighborhood of the singular points of the singular fibers.
Next, we pick a $c$-invariant symplectic form on the base, $\omega_B$, so that 
$\omega'=pr^* \omega_B$ is inducing the same orientation on the normal bundle of the fibers as $\omega_0$ and $\omega_1$. 
Then, we observe that
the forms $\omega_0+s\omega'$,
 $\omega_1+s\omega'$ with $s\geq0$ and, for a sufficiently big $s^*>0$, the forms
 $a\omega'+t{\cdot}\omega_1+(1-t){\cdot}\omega_0$ with $a\geq s^*$, $t\in [0,1]$,
 are also symplectic forms on $X$.

 This path connecting $\omega_0$ with $\omega_1$ (through $\omega_0+a\omega'$ and $\omega_1+a\omega'$) is $c$-equivariant, but to apply Moser's argument
 we need to make it of a constant cohomology class. To handle this, we use Theorem \ref{NaMo-Omega}. 
 Here, a specific for real geometry difficulty arises:
 we can not avoid appearance of $J_t$-holomorphic $(-2)$-curves at a discrete set of values of $t$. 
 To overcome such a difficulty we replace the path $(J_t,\omega_t)$ by a chain of $c$-equivairiant pathes $(J_t^k, \omega_t^k), k\in \{0,\dots, r\}, 0\le t\le 1$, so that 
all $\omega_0^k, \omega_1^k$ are of the same cohomology class, $(J_0^0, \omega_0^0)= (J_0,\omega_0)$, $(J_1^r, \omega_1^r)= (J_1,\omega_1)$, and $(J_1^k,\omega_1^k)$ is $c$-equivariant isomorphic to $(J_0^{k+1},\omega_0^{k+1})$. Our construction goes as follows. We put the initial path in a generic position with respect to the divisors responsable in the space of almost complex structures for appearance of $(-2)$-curves. Then, at each instant $t_0$ where a $c$-equivariant $J_{t_0}-$holomorphic $(-2)$-curve appears, we approximate the almost complex structures $J_t$ with $t$ close to $t_0$ by structures that are integrable in a neighborhood of  the curve and represent locally one of standard Atiyah-flop families \cite{A} of an appropriate signature:
$x^2+y^2+z^2=- (t-t_0)^2$ if $c$ has no fixed point on the $(-2)$-sphere,  and $x^2+y^2-z^2=- (t-t_0)^2$ otherwise. Recall that over $t\ne t_0$ the latter have a symmetry permuting $t=t_0+\epsilon$ and $t=t_0-\epsilon$. Next, taking $\epsilon> 0$
sufficiently small we deform $J_{t+\epsilon}$ and $J_{t-\epsilon}$ outside the inserted Atiyah-flop part to achieve a full symmetry between $J_{t+\epsilon}$ and $J_{t-\epsilon}$.
After that there remain to apply the inflation procedure to $\omega_{t+\epsilon}$ and  $\omega_{t-\epsilon}$ symmetrically for achieving the cohomology class relation
 $[\omega_{t+\epsilon}]=[\omega_{t-\epsilon}]=[\omega_0]$ .
\end{proof}

\begin{remark}
 For all but one, with $m=3$, $c$-minimal real 
 rational ruled symplectic 4-manifolds there are no $c$-equivariant $(-2)$-curves at all, and by this reason,
 for all but this one, we get even a $c$-equivariant isotopy in Theorem \ref{main-conic}.
\end{remark}

The cases when $X$ is $\C P^2$ or $\C P^1\times \C P^1$ (items (1) and (2) in Corollary \ref{cases})
can be easily reduced to the case with a ruling: 

\begin{itemize}
\item 
If $X$ is $\C P^1\times \C P^1$
and $X^c$  is either $\R P^1\times \R P^1$ or $\emptyset$, then $(X, \omega, c)$ has a $c$-equivariant rational ruling
without singular fibers. 
\item If $X=\C P^1\times \C P^1$
and $X^c= S^2$,  
it acquirs a $c$-equivariant rational ruling with two singular fibers after a blow up at a pair of $c$-conjugate points, $p$ and $c(p)\ne p$. 
\item Similarly, if $X=\C P^2$ then $(X, \omega, c)$ acquirs a $c$-equivariant rational ruling without singular fibers after a blow up at a $c$-invariant point, $p=c(p)$.
\end{itemize}

The last two cases (items (4) and (5) in Corollary \ref{cases}):  $K^2=2$ with $X^c= 4S^2$ and $K^2=1$ with $X^c=4S^2\sqcup \R P^2$, are more intricate. Here, instead of symplectic $(-1)$-surgeries
(that is the blow-ups at points) we use $(-2)$- and $(-4)$-surgeries
(that is symplectic blow-ups of Lagrangian spheres and, respectively, Lagrangian projective planes, see \cite{Le}).

To treat the case $K^2=2, X^c=4S^2$ we perform a $(-2)$-surgery at one of the four spherical components of $X^c$:

 \begin{enumerate}
 \item
 This surgery 
 replaces a neighborhood of
such a Lagrangian sphere $S^2\subset X^c$ modeled on its cotangent bundle $T^*(S^2)$
by a neighborhood of a symplectic $(-2)$-sphere $S^2_{new}=\C P^1$ modeled on its $T^*(\C P^1)=O(-2)$ line bundle.
More precisely, we cut out a neighborhood symplectomorphic to a tubular $\epsilon$-neighborhood in $T^*(S^2)$, replace it by the standard resolution
$\phi : \tilde U\to U$ of $U\subset Q=\{\sum_{j=1}^3 z_j^2 \}\subset \C^3, U=Q\cap \{\sum_{j=1}^3 \vert z_j\vert ^2\le \sqrt{2\epsilon}\}$ which we equip with an $SO(3)$-invariant 
symplectic structure obtained from $\phi^*(\frac{i}2\sum^3_{j=1} dz_j\wedge\overline{dz_j})$ by a small perturbation in a smaller neighborhood. 
 \item This surgery is compatible with anti-involutions, and the resulting real symplectic 4-manifold, $(X',\omega',c')$,
turns out to have a $c'$-equivariant rational ruling with $m=3$ where $S^2_{new}$ is a $c'$-invariant bi-section (without $c'$-fixed points) and 
the fibers 
split off from the representatives of $c_1(X')$ intersecting $S^2_{new}$.
\end{enumerate}

After that the same arguments as in our treatment of manifolds endowed with a rational ruling apply and prove that there exists an integrable structure $J'$, such that:
 \begin{itemize}
  \item $(X',J',c')$ is real algebraic.
  \item $\omega'$ is $J'$-K\"ahler.
  \item $S^2_{new}$ is $J'$-holomorphic.
 \end{itemize}
Finally, we show that the result of the inverse $(-2)$-surgery at $S^2_{new}$ is also K\"ahler and, due
to the symplectic uniqueness of the $(-2)$-surgery, equivariant symplectomorphic to $(X,\omega, c)$ we have started from.

The 
proof in the 
case $K^2=1$, $X^c=4S^2\sqcup \R P^2$ 
follows the same lines. We indicate only the difference in surgeries. This time, we use a $(-4)$-surgery (also known as a \emph{rational blow-up}) which we apply
to the $\R P^2$-component of $X^c$. This surgery transforms  a neighborhood of $\R P^2\subset X^c$ symplectomorphic to a tubular $\epsilon$-neighborhood in $T^*(\R P^2)$ (which can be seen as the quotient of a neighborhood in $T^*(S^2)$ by the involution induced by the antipodal map $S^2\to S^2$)
into a neighborhood of $S^2_{new}=\C P^1$ in its $O(-4)$ line bundle (which can be obtained by picking up a neighborhood of $\C P^1$ in $O(-2)$ used for $(-2)$-surgery
and taking the quotient under multiplication by $-1$ in $O(-2)$). This surgery is also compatible with anti-involutions.
The resulting symplectic 4-manifold with an anti-involution, $(X',\omega',c')$, has a $c'$-equivariant rational ruling with $m=4$ where the $(-4)$-sphere $S^2_{new}$
produced by the surgery is a $c'$-invariant bi-section (without $c'$-fixed points) and the fibers 
split off from representatives of $c_1(X')$ intersecting $S^2_{new}$.

\bibliographystyle{amsplain}

\end{document}